\documentclass[12pt, reqno]{amsart}
\usepackage{amsmath, amsthm, amscd, amsfonts, amssymb}
\usepackage{bm}
\usepackage{graphicx}
\graphicspath{{images/}}
\usepackage{comment}
\usepackage{amscd}
\usepackage{mathrsfs} 
\usepackage[shortlabels]{enumitem} 

\newcommand{\bea}{\begin{eqnarray}}
\newcommand{\eea}{\end{eqnarray}}

\newcommand{\clb}{\mathcal{B}}

\newcommand{\clh}{\mathcal{H}}

\newcommand{\T}{\mathbb{T}}

\newcommand{\D}{\mathbb{D}}

\newcommand{\C}{\mathbb{C}}
\newcommand{\R}{\mathbb{R}}

\def\textmatrix#1&#2\\#3&#4\\{\bigl({#1 \atop #3}\ {#2 \atop #4}\bigr)}
\def\dispmatrix#1&#2\\#3&#4\\{\left({#1 \atop #3}\ {#2 \atop #4}\right)}
\newcommand{\be}{\begin{equation}}
	\newcommand{\ee}{\end{equation}}
\newcommand{\ben}{\begin{eqnarray*}}
	\newcommand{\een}{\end{eqnarray*}}

\newcommand{\NI}{\noindent}

\newcommand{\bi}{\begin{itemize}}
	\newcommand{\ei}{\end{itemize}}

\newtheorem{Question}{\sc Question}

\theoremstyle{definition}

\newtheorem*{theorem*}{Theorem}
\theoremstyle{plain}

\newtheorem{thm}{Theorem}[section]

\newtheorem{lem}[thm]{Lemma}
\newtheorem{prop}[thm]{Proposition}

\theoremstyle{definition}

\newtheorem{ex}[thm]{Example}

\numberwithin{equation}{section}

\let\phi=\varphi

\begin{document}

\setcounter{page}{1}
	
\title[Invertibility of Bergman Toeplitz operators]{Invertibility of Bergman Toeplitz operators}

\author[Javed]{Mo Javed}
\address{Indian Institute of Technology Roorkee, Department of Mathematics,
		Roorkee-247 667, Uttarakhand,  India}
\email{mo\_j@ma.iitr.ac.in, javediitr07@gmail.com}

\author[Maji]{Amit Maji}
\address{Indian Institute of Technology Roorkee, Department of Mathematics,
		Roorkee-247 667, Uttarakhand,  India}
\email{amit.maji@ma.iitr.ac.in, amit.iitm07@gmail.com ({Corresponding author)}}



\begin{abstract}
In this paper, we establish the invertibility of the Berezin transform of the symbol 
as a necessary and sufficient condition for the invertibility of the Toeplitz operator 
on the Bergman space $L^2_a(\D)$. More precisely, if $\phi=cg+d\Bar{g}$, where $c,d\in\C$ and $g\in H^{\infty}(\D)$, the space of all bounded analytic functions, then $T_{\phi}$ is invertible on $L^2_a(\D)$ if and only if $\inf\limits_{z\in\D}\left|\widetilde{\,{\phi}}(z)\right|=\inf\limits_{z\in\D}|\phi(z)|>0$, where $\widetilde{\,{\phi}}$ is the Berezin transform of $\phi$.
\end{abstract}	
\maketitle

\section{Introduction}\label{Sec:Introduction}

Let $\D=\{z\in \C : |z|<1\}$ be the open unit disc and $\T$ be the unit circle in the complex plane $\C$. Let $dA$ denote the normalized Lebesgue area measure on $\D$. The space $L^2(\D, dA)$ stands for the complex Hilbert space of square integrable functions on $\D$ with the inner product defined by
\[
\langle f, g \rangle=\int_{\D}f(z)\overline{g(z)}\,dA(z) \quad (f,g\in L^2(\D,dA)).
\]
The \emph{Bergman space}, denoted by $L^2_a(\D)$, is the closed subspace of all analytic functions on $\D$ in $L^2(\D, dA)$. Moreover, the space $L^2_a(\D)$ is a reproducing kernel Hilbert space with respect to the kernel function $K_z$ given by
\[
K_{z}(w)= K(w,z)= \frac{1}{(1-\Bar{z}w)^2} \quad (z,w \in \D).
\] 
Let $H^2(\D)$ be the Hardy space of analytic functions $f$ on $\D$ such that 
\[
\| f \| =\sup_{0 \leq r <1}\left(\frac{1}{2\pi}\int_0^{2\pi} | f(re^{it})|^2\, dt\right)^{1 \over 2} <\infty.
\]
We denote by $L^{\infty}(\D)$ (or $L^{\infty}(\T)$) the Banach space of essentially bounded functions on $\D$ (or $\T$) equipped with the sup-norm $\|\cdot\|_{\infty}$. For $\phi\in L^{\infty}(\D)$, the Toeplitz operator $T_{\phi}$ on the \emph{Bergman space} is defined by
\[
T_{\phi}f=P(\phi f), \quad (f\in L^2_a(\D)),
\]
where $P$ is the orthogonal projection of $L^2(\D, dA)$ onto $L^2_a(\D)$. For $\phi\in L^{\infty}(\D)$, the Berezin transform of $\phi$ on $\D$, denoted by $\widetilde{\phi}$,  is defined as follows
\[
\widetilde{\phi}(z)=\langle T_{\phi}k_z, k_z\rangle =  \int_{\D}\phi(w)\frac{(1-|z|^2)^2}{\left|1- \bar{z} w\right|^4}\,dA(w) \quad (z \in\D),
\]
where $k_z$ is the normalized reproducing kernel, i.e., $k_z=\dfrac{K_z}{\|K_z\|}$. On the other hand, for $\phi\in L^{\infty}(\T)$, the harmonic extension of the function $\phi$ in $\D$, denoted by $\breve{\phi}$, is given by the Poisson integral formula:
\[
\breve{\phi}(z)=\frac{1}{2\pi}\int_{0}^{2\pi}\phi(e^{it})\frac{1-|z|^2}{\left|1-z e^{-it}\right|^2}\,dt \quad (z \in\D).
\]
It is easy to see that the harmonic extension of a function $\phi\in L^{\infty}(\T)$ in $\D$ coincides with the Berezin transform of the Toeplitz operator $T_{\phi}$ with symbol $\phi$ on the Hardy space $H^2(\D)$ over the unit disc. The main purpose of this note is to study the invertibility criteria for the certain Toeplitz operators on the Bergman space.

One of the key problems in operator theory and function theory is to determine the invertibilty of Toeplitz operators. In this direction, the complete characterization of invertible Toeplitz operators on $H^2(\D)$ was studied by Devinatz \cite{Devinatz-TOEP. OPER. ON H2 SPACES}. Douglas \cite{Douglas-BANACH ALGEBRA TECH. IN THEORY OF TOEP. OPER.} characterized the invertibility of Toeplitz operators on $H^2(\D)$ with continuous symbols by showing that for $\phi\in C(\T)$, the space of continuous functions on $\T$, the Toeplitz operator $T_{\phi}$ is invertible on $H^2(\D)$ if there exists $\delta>0$ such that $|\breve{\phi}(z)|\geq \delta$ for all $z\in\D$. Thereby, Douglas raised the following problem:
\textsf{Let $\phi\in L^{\infty}(\T)$. Suppose there exists a $\delta>0$ such that $|\breve{\phi}(z)|\geq \delta$ for all $z\in\D$. Is the Toeplitz operator $T_{\phi}$ invertible on $H^2(\D)$?}

\NI
Many researchers including Chang and Tolokonnikov \cite{Tolokonnikov-ESTIMATES IN THE CARLESON CORONA THM.} investigated the above problem and obtained a sufficient condition that there exists a positive $\delta_0 ~(<1)$ such that if $\delta_0<|\breve{\phi}(z)|<1$ for all $z\in\D$, then $T_{\phi}$ is invertible on $H^2(\D)$ (see also \cite{NIKOLSKI-TREATISE ON SHIFT}). On the other hand, Wolff \cite{WOLFF-COUNTEREXAMPLES} constructed a function in $L^{\infty}(\T)$ that serves as an elegant counterexample for a negative answer to Douglas's question. Thus, in the Bergman space setting the question reads:

\begin{Question}\label{Q:Douglas question for Bergman space}
\textsf{Is the invertibility of Bergman Toeplitz operator $T_{\phi}$ for $\phi\in L^{\infty}(\D)$ equivalent to the invertibility of the Berezin transform $\widetilde{\phi}$ in $L^{\infty}(\D)$?}
\end{Question}

The above question is incredibly challenging and even for a bounded harmonic function $\phi$ on $\D$, the problem is still unsolved. However, the affirmative answers are available in the literature in the case when $\phi\in L^{\infty}(\D)$ is either `real harmonic', or `non-negative', or `analytic', or `coanalytic'. For a detailed survey, see \cite{LUECKING-INEQUALITIES ON BERGMAN SPACES,McDONALD-SUNDBERG-TOEPLITZ OPER. ON DISC,Zhao_Zheng}. The problem of invertibility for Bergman Toeplitz operators with harmonic symbols had also been explored by Yoneda \cite{Yoneda-ON INVERTIBILITY OF BERGMAN TOEP. OPER.}. He obtained the invertibility of Berezin transform of the harmonic symbol of the form $\phi=cg+d\overline{g}$, where $c,d\in\C$ and $g\in H^{\infty}(\D)$ as a sufficient condition for the invertibility of the Bergman Toeplitz operator $T_{\phi}$. Moreover, for such a $\phi$ if $g\in H^{\infty}(\D)\cap C(\overline{\D})$, then he obtained an affirmative answer to the Question \ref{Q:Douglas question for Bergman space} (see \cite[Theorem 2.6, Theorem 2.8]{Yoneda-ON INVERTIBILITY OF BERGMAN TOEP. OPER.}). For a thorough discussion on Bergman Toeplitz operators, one can see \cite{HAKAN-KORENBLUM-ZHU,KORENBLUM-ZHU-AN APPLICATION OF TAUBERIAN,ZHU-OPER. THEORY IN FUNCTION SPACES}.

In this paper we have obtained an affirmative answer to Question \ref{Q:Douglas question for Bergman space} for harmonic symbols of the form $cg+d\overline{g}$, where $g\in H^{\infty}(\D)$ which relaxes the condition that $g$ be continuous on $\overline{\D}$ unlike Yoneda's result. Along the way, we prove some characterization results in the general setting of Hilbert space. In turn, it precisely extends the results concerning the invertibility of $T_{\phi}$ in terms of the invertibility of the Berezin transform of its symbol when $\phi$ is either real harmonic, analytic, coanalytic, or $T_{\phi}$ is normal provided $\phi$ is bounded harmonic.

The rest of the paper is structured as follows. Section \ref{Sec:Preliminaries} focuses on establishing the necessary background, and along the way, we fix some notations and definitions. In Section 3, we prove the invertibility of the Berezin transform of certain harmonic symbols determines the invertibility of the Bergman Toeplitz operator which extends Yoneda's result.

\section{Preliminaries}\label{Sec:Preliminaries}

This section compiles notations, definitions, and some basic results that are used in this paper. We denote $\clh$ as a Hilbert space over the field of complex numbers 
$\mathbb{C}$ and $\clb(\clh)$ as the $C^*$-algebra of all bounded linear operators on 
$\clh$. A linear operator $T$ is said to be bounded below if there exists $c> 0$ such that $\|Th \| \geq c \| h \|$ for all $h \in \clh$. An operator $T \in \clb(\clh)$ is said to be hyponormal if $T^*T -TT^*$ is a positive operator, equivalently, $\|Th\| \geq \|T^*h\| $ for all $h \in \clh$.

Let $H(\D)$ be a reproducing kernel Hilbert space (for e.g. $H^2(\D)$, $L^2_a(\D)$) on 
$\D$ and the map $K: \D \times \D \rightarrow \mathbb{C} $ 
defined by
\[
K(w,z) = K_z(w)= \langle K_z, K_w \rangle
\]
be the reproducing kernel function of $H(\D)$. We denote $k_z=\dfrac{K_z}{\|K_z\|}$ \
for $z \in \D$ as the normalized reproducing kernel of $H(\D)$ and the set $\{{k}_z: z \in \D \}$ is a total set in $H(\D)$. For more details and references on RKHS, see \cite{PR-rkHs-Book}. For a bounded operator $T$ on $H(\D)$, the {\textit{Berezin transform}} of $T$, denoted by $\widetilde{T}$, is a complex-valued function on $\D$ defined as
\[
\widetilde{T}(z)=\langle Tk_z,k_z \rangle \quad \text{for all }z\in\D.
\]
Thus, every bounded linear operator $T$ on $H(\D)$ induces a bounded function 
$\widetilde{T}$ on $H(\D)$. Indeed, $\| \widetilde{T} \|_{\infty} \leq \| T \|$. Therefore, the {\textit{Berezin transform}} of an operator yields crucial information about the operator. In this note we mainly focus on the {\textit{Berezin transform}} of Toeplitz operators on the Bergman space to determine the invertibility criterion of the Toeplitz operators. Throughout this article, $H^\infty(\D)$ stands for the algebra of bounded analytic functions on $\D$ and it is also a multiplier algebra for the Hardy space as well as the Bergman space, and $C(\T)$ is the space of all continuous functions on $\T$.

We will be using the following well-known results very often.

\begin{prop}\label{Prop:Properties of T_phi}
Let $\phi, \psi\in L^{\infty}(\D)$ and $g\in H^{\infty}(\D).$ For $c, d\in \C,$ the following results hold for the Bergman Toeplitz operators
\begin{enumerate}
		\item $T_{c\phi+d\psi}=cT_{\phi}+dT_{\psi}$.
		\item $T_{\phi}^*=T_{\overline{\phi}}$.
		\item $T_{\phi}T_{g}=T_{\phi g}$.
		\item $T_{\overline{g}}T_{\phi}=T_{\overline{g}\phi}$.
\end{enumerate}
\end{prop}

\begin{lem}\label{Lem:A invertile iff A and A* bounded below}
Suppose $T$ is a bounded operator on a Hilbert space $\clh$. Then 
\begin{enumerate}
        \item $T$ is invertible if and only if both $T$ and $T^*$ are bounded below.
        \item $T$ is bounded below if and only if $T^*$ is surjective.
\end{enumerate}
\end{lem}

\begin{thm}[cf.\ \cite{McDONALD-SUNDBERG-TOEPLITZ OPER. ON DISC}]\label{Thm:analytic and coanalytic invertibility via Berezin transform}
For $\phi\in H^{\infty}(\D)$, let $T_{\phi}:L^2_a(\D)\to L^2_a(\D)$ be the Toeplitz operator. Then the following are equivalent
\begin{enumerate}
        \item $T_{\phi}$ is invertible.
        \item $T_{\overline{\phi}}$ is invertible.
        \item $\inf\limits_{z\in\D}|\phi(z)|>0$.
\end{enumerate}
\end{thm}

\begin{thm}[cf.\ \cite{Zhao_Zheng}]\label{Thm:invertible normal Toeplitz operator with harmonic symbol}
Let $\phi\in L^{\infty}(\D)$ be a bounded harmonic function on $\D$ such that $T_{\phi}$ is a normal operator on $L^2_a(\D)$. Then $T_{\phi}$ is invertible if and only if 
$\inf\limits_{z\in\D}|\widetilde{\,{\phi}}(z)|=\inf\limits_{z\in\D}|\phi(z)|>0$.
\end{thm}

\section{Toeplitz operator on the Bergman space}

In this section we discuss the invertibility of Bergman Toeplitz operators with the help of the Berezin transform of the symbol. More precisely, Theorem \ref{Thm:|s|<1 iff bounded below} gives an affirmative answer to the Problem 2.3 stated in \cite{Yoneda-ON INVERTIBILITY OF BERGMAN TOEP. OPER.} in a more general setting. Consequently, we obtain a refined characterization of invertible Bergman Toeplitz operators for some certain bounded harmonic functions.

\begin{thm}\label{Thm:|s|<1 iff bounded below}
Let $T\in\clb(\clh)$ be a hyponormal operator and $s\in \C$ with $|s|<1$. Then the following two are equivalent
\begin{enumerate}
\item $T^*$ is bounded below.
\item $sT+T^*$ is bounded below.
\end{enumerate}
\end{thm}

\begin{proof}
Suppose $T^*$ is bounded below on $\clh$. Since $T$ is hyponormal on $\clh$, $\|Th \| \geq \|T^* h \|$ for all $h \in \clh$, and hence $T$ is also bounded below. 
Now Lemma \ref{Lem:A invertile iff A and A* bounded below} infers that $T$ is invertible. If $f\in \clh$ and $f\neq 0$, then we choose $0\neq h\in \clh$ such that $f=Th$. For
$s\in \C$ with $|s|<1$, consider 
\begin{align*}
\|(sT+T^*)f\| 
        &=\|(sT+T^*)Th\| \\
        &=\sup\{|\langle (sT+T^*)Th, h'\rangle| : h'\in \clh,\, \|h'\|=1\} \\
        &\geq \left|\left\langle (sT+T^*)Th, \frac{h}{\|h\|}\right\rangle\right| \\
        &=\frac{1}{\|h\|}|\langle sT^2h + T^*Th, h\rangle| \\
        &=\frac{1}{\|h\|}|s\langle Th, T^*h\rangle + \|Th\|^2| \\
        &\geq\frac{1}{\|h\|}\left(\|Th\|^2-|s||\langle Th, T^*h \rangle|\right)\\
        &\geq\frac{1}{\|h\|}(\|Th\|^2-|s|\|Th\| \|T^*h\|) 
        \tag{By Cauchy-Schwarz inequality} \\
        &\geq\frac{1}{\|h\|}(\|Th\|^2-|s|\|Th\|^2) \tag{By hyponormality of $T$} \\
        &=\frac{(1-|s|)}{\|h\|}\|Th\|^2.
\end{align*}
Now $h=T^{-1}f$. Thus, $0<\|h\|\leq\|T^{-1}\|\|f\|$. Therefore
\begin{align*}
\|(sT+T^*)f\|
&\geq\frac{(1-|s|)}{\|T^{-1}\|\|f\|}\|f\|^2=\frac{(1-|s|)}{\|T^{-1}\|}\|f\| \qquad (f\in \clh).
\end{align*}
Hence $sT+T^*$ is bounded below.

Conversely, assume that $sT+T^*$ is a bounded below operator on $\clh$. Suppose, if possible, that $T^*$ is not bounded below. Then there exists a sequence $(f_n)_{n\geq 1}$ in $\clh$ with $\|f_n\|=1$ for every $n\geq 1$ and $\|T^*f_n\|\to 0$ as $n\to\infty$. Recall that an operator $A$ on $\clh$ is bounded below if and only if $A^*$ is surjective. Therefore, for every $n\geq 1$, there exists $h_n\in\clh$ such that 
\[
(\Bar{s}T^*+T)h_n=f_n.
\]
Also,
\begin{align*}
\|T^*f_n\|
        &=\|T^*(\Bar{s}T^*+T)h_n\| \\
        &=\|\Bar{s}T^{*2}h_n+T^*Th_n\| \\
        &\geq \left|\left\langle \Bar{s}T^{*2}h_n+T^*Th_n, \frac{h_n}{\|h_n\|} 			 \right\rangle\right| \\
        &=\frac{1}{\|h_n\|}\left|\Bar{s}\langle T^*h_n, Th_n \rangle + \|Th_n\|^2\right| \\
        &\geq\frac{1}{\|h_n\|}(\|Th_n\|^2-|s|\|Th_n\| \|T^*h_n\|) 
        \tag{By Cauchy-Schwarz inequality} \\
        &\geq\frac{1}{\|h_n\|}(\|Th_n\|^2-|s|\|Th_n\|^2) \tag{By hyponormality of $T$} \\
        &=\frac{(1-|s|)}{\|h_n\|}\|Th_n\|^2 \\
        &=(1-|s|)\|h_n\|\left\|T\left(\frac{h_n}{\|h_n\|}\right)\right\|^2.
\end{align*}
Now
\[
1=\|f_n\|=\|(\Bar{s}T^*+T)h_n\|\leq\|\Bar{s}T^*+T\|\|h_n\| ~~\text{for every}~~ n \geq 1. 
\]
Thus,
\begin{align*}
        \|T^*f_n\|
        \geq \frac{(1-|s|)}{\|\Bar{s}T^*+T\|}\left\|T\left(\frac{h_n}{\|h_n\|}\right)\right\|^2
        \geq \frac{(1-|s|)}{\|\Bar{s}T^*+T\|}\left\|T^*\left(\frac{h_n}{\|h_n\|}\right)\right\|^2 \qquad (n\geq 1).
\end{align*}
Thus it follows that
\[
\left\|T\left(\frac{h_n}{\|h_n\|}\right)\right\|\to 0 \text{ and }
\left\|T^*\left(\frac{h_n}{\|h_n\|}\right)\right\|\to 0 \text{ as } n\to\infty.
\]
Since $sT+T^*$ is bounded below, we can choose $M>0$ such that $M\|h_n\|\leq\|(sT+T^*)h_n\|$ for every $n\geq 1$. Therefore
\[
0<M\leq\left\|sT\left(\frac{h_n}{\|h_n\|}\right)+T^*\left(\frac{h_n}{\|h_n\|}\right)\right\|\to 0 \text{ as } n\to\infty,
\]
which leads to a contradiction that $sT+T^*$ is bounded below. This establishes $(2) \implies (1)$.

\end{proof}

\begin{thm}\label{Thm:|s|>1 iff bounded below}
Let $T\in\clb(\clh)$ be a hyponormal operator and $s\in \C$ with $|s|>1$. The following two are equivalent
\begin{enumerate}
\item $T$ is bounded below.
\item $sT+T^*$ is bounded below.
\end{enumerate}
\end{thm}

\begin{proof}
For $h \in\clh$, 
\begin{align*}
(|s|+1)\|Th\|
        &=|s|\|Th\|+\|Th\| \\
        &\geq|s|\|Th\|+\|T^*h\| \tag{By hyponormality of $T$} \\
        &\geq\|sTh+T^*h\|=\|(sT+T^*)h\| \\
        &\geq|s|\|Th\|-\|T^*h\| \\
        &\geq(|s|-1)\|Th\|.
\end{align*}
Now $|s| -1>0$ as $|s|>1$. Therefore, the above chain of inequalities establishes that 
$(1) \iff (2)$.
\end{proof}

\begin{ex}
Consider the shift operator $T_z$ on the Hardy space $H^2(\D)$. It is an analytic Toeplitz operator so is hyponormal. Also, $T_z$ is bounded below, in fact, $\|T_zh\|=\|h\|$ for all $h\in H^2(\D)$. If $|s|>1$ then it follows from Theorem \ref{Thm:|s|>1 iff bounded below} that $sT_z+T_z^*$ is also bounded below. Note that $T_z$ is not 
surjective and hence by Lemma \ref{Lem:A invertile iff A and A* bounded below}, 
$T_z^*$ is not bounded below. This shows that $T^*$ may not be bounded below while $sT+T^*$ is bounded below for $|s|>1$.  
\end{ex}

From the above theorems, we have the following result.

\begin{prop}\label{Prop: T iff invertible}
Let $T\in\clb(\clh)$ be hyponormal on $\clh$ and $s\in\C$ with $|s|\neq 1$. Then the following are equivalent
\begin{enumerate}
\item $T$ is invertible on $\clh$.
\item $sT+T^*$ is invertible on $\clh$.
\end{enumerate}
\end{prop}

\begin{proof}
By virtue of Lemma \ref{Lem:A invertile iff A and A* bounded below}, it suffices to show that both $T$ and $T^*$ are bounded below if and only if both $sT+T^*$ and $\Bar{s}T^*+T$ are bounded below. First we note that if $s=0$, then it is trivial as $T$ is invertible if and only if $T^*$ is invertible. Now assume that $s\neq 0$.\\
    
\NI     
{\bfseries Case (i):} Suppose $|s|>1$.

Let both $T$ and $T^*$ be bounded below. It follows from Theorem \ref{Thm:|s|>1 iff bounded below} that $sT+T^*$ is bounded below. Moreover,
\[
(sT+T^*)^* = \Bar{s}T^*+T=\Bar{s}\left(\frac{1}{\Bar{s}}T+T^*\right) \text{ with } \left|\frac{1}{\Bar{s}}\right|=\frac{1}{|s|}<1.
\]
Thus, Theorem \ref{Thm:|s|<1 iff bounded below} implies that $\Bar{s}T^*+T$ is bounded below as well. 

On the other hand, if both $sT+T^*$ and $\Bar{s}T^*+T$ are bounded below, then $T$ is bounded below follows from Theorem \ref{Thm:|s|>1 iff bounded below}. However, the relation $\frac{1}{\Bar{s}}T+T^*=\frac{1}{\Bar{s}}(\Bar{s}T^*+T)$ along with Theorem \ref{Thm:|s|<1 iff bounded below} yields $T^*$ to be bounded below too.
    
\NI     
{\bfseries Case (ii):} Suppose $0<|s|<1$.

Using Theorems \ref{Thm:|s|<1 iff bounded below} and \ref{Thm:|s|>1 iff bounded below} and the relation
\[
\Bar{s}T^*+T=\Bar{s}\left(\frac{1}{\Bar{s}}T+T^*\right) \text{ with } \left|\frac{1}{\Bar{s}}\right|=\frac{1}{|s|}>1, 
\]
yields that both $T$ and $T^*$ are bounded below if and only if both $sT+T^*$ and 
$\Bar{s}T^*+T$ are bounded below.
\end{proof}

We are now in a position to present the main result of this section.

\begin{thm}\label{Thm:main theorem}
Let $\phi=cg+d\overline{g}$, where $g\in H^{\infty}(\D)$ and the constants $c,d\in\C$, then the following are equivalent
\begin{enumerate}
\item $T_{\phi}$ is invertible on $L^2_a(\D)$.
\item There exists $\delta>0$ such that $|\phi(z)|\geq \delta$ for all $z\in\D$, i.e., $\inf\limits_{z\in\D}|\widetilde{\,{\phi}}(z)|=\inf\limits_{z\in\D}|\phi(z)|>0$.
\end{enumerate}
\end{thm}

\begin{proof}
If either $c=0$ or $d=0$, then the equivalence of (1) and (2) follows from Theorem \ref{Thm:analytic and coanalytic invertibility via Berezin transform}. Assume now that $c\neq0$ and $d\neq0$ and let $s=c/d$, then $T_{\phi}=dT_{sg+\overline{g}}$.
    
\NI    
{\bfseries Case (i):} For $|s|=1$. 

Firstly, $T_{\phi}$ is normal on $L^2_a(\D)$. Indeed,
\begin{align*}
        T_{\phi}^*T_{\phi}
        &=|d|^2(\Bar{s}T_{\overline{g}}+T_g)(sT_{g}+T_{\overline{g}})\\    
        &=|d|^2(T_{\overline{g}}T_{g}+sT_{g}T_{g}+\Bar{s}T_{\overline{g}}
        T_{\overline{g}}+T_{g}T_{\overline{g}})
        \\
        &=|d|^2(sT_{g}+T_{\overline{g}})(\Bar{s}T_{\overline{g}}+T_g)\\
        &=T_{\phi}T_{\phi}^*.
\end{align*}
Hence, Theorem \ref{Thm:invertible normal Toeplitz operator with harmonic symbol} yields the equivalence of $(1)$ and $(2)$.

\NI    
{\bfseries Case (ii):} When $|s| \neq 1$.
    
Note that for $g\in H^{\infty}(\D)$, $T_g$ is hyponormal because 
for $f\in L^2_a(\D)$,
\[
\|T^*_{g}f\| = \|T_{\overline{g}}f\|=\|P(\overline{g}f)\|\leq\|\overline{g}f\|=\|gf\|=\|T_{g}f\|.
\]
Also, for $|s|\neq1$ we have
\[
||s|-1||g(z)|=\left||sg(z)|-\left|\overline{g(z)}\right|\right|\leq\left|sg(z)+\overline{g(z)}\right|\leq(|s|+1)|g(z)|, \qquad (z\in\D).
\]
Thus 
\begin{equation}\label{Eq:Equivalence of g and sg+gBar}
||s|-1|\left(\inf_{z\in\D}|g(z)|\right)\leq\inf_{z\in\D}\left|sg(z)+\overline{g(z)}\right|\leq(|s|+1) \left(\inf_{z\in\D}|g(z)|\right).
\end{equation}
Let $T_{\phi}$ be invertible on $L^2_a(\D)$. Thus, $T_{sg+\overline{g}} = sT_{g}+T^*_{g}$ is invertible. Employing Proposition \ref{Prop: T iff invertible}, we have $T_{g}$ is invertible on $L^2_a(\D)$. It now follows from Theorem \ref{Thm:analytic and coanalytic invertibility via Berezin transform} that $\inf\limits_{z\in\D}|g(z)|>0.$ Hence the inequality \eqref{Eq:Equivalence of g and sg+gBar} implies that 
\[
\inf_{z\in\D}|\phi(z)|=|d|\inf_{z\in\D}\left|sg(z)+\overline{g(z)}\right|>0.
\]

Conversely, assume that $\inf\limits_{z\in\D}|\phi(z)|>0$. Again, the inequality \eqref{Eq:Equivalence of g and sg+gBar} and Theorem \ref{Thm:analytic and coanalytic invertibility via Berezin transform} infers that $T_{g}$ is invertible on $L^2_a(\D)$. Finally, Proposition \ref{Prop: T iff invertible} yields that $T_{\phi}$ is invertible on $L^2_a(\D)$. This completes the proof.
\end{proof}

We now give a concrete example to supplement our results.

\begin{ex}
For	fix $t\in\R$ and consider the function 
\[
\phi(z)=\left(\frac{1+z}{1-z}\right)^{it} \quad (z\in\D).
\]
Clearly, $\phi\notin C(\overline{\D})$ and it is known that $\phi\in H^{\infty}(\D)$. Note that if $1+z=re^{i\theta}$ and $1-z=\rho e^{i\eta}$, where $0<r,\rho <2$  and 
$-\pi/2 < \theta, \rho < \pi/2$ then 
\[
\left|(1+z)^{it}\right|=\left| e^{it\log(1+z)} \right|=\left| e^{it(\log r+i\theta)} \right|=e^{-t\theta}\geq e^{-|t\theta|}\geq e^{-|t|\pi/2}>0.
\] 
Similarly,
\[
\left|(1-z)^{it}\right|=\left| e^{it\log(1-z)} \right|=\left| e^{it(\log \rho+i\eta)} \right|=e^{-t\eta}\geq e^{-|t\eta|}\geq e^{-|t|\pi/2}>0.
\]
Consequently, the Toeplitz operators $T_{(1+z)^{it}}$ and $T_{(1-z)^{it}}$ are invertible on $L^2_a(\D)$. Now we write 
\[
(1-z)^{it}\phi(z)=(1+z)^{it} \quad (z\in\D).
\]
Therefore,
\[
T_{(1+z)^{it}}T_{\phi}=T_{(1-z)^{it}}
\]
from which it follows that $T_{\phi}$ is invertible on $L^2_a(\D)$. In this case,
$\inf\limits_{z\in\D}|\phi(z)|>0$.
\end{ex}


\vspace{0.2 cm}	
	


\begin{thebibliography}{99}

    \bibitem{Devinatz-TOEP. OPER. ON H2 SPACES}{A.~Devinatz,}
    {\it Toeplitz operators on $H^2$ spaces,}
    {Trans.\ Amer.\ Math.\ Soc.\ $\bm{112}$ (1964), no.\ 2, 304--317.}



    \bibitem{Douglas-BANACH ALGEBRA TECH. IN THEORY OF TOEP. OPER.}{R.~Douglas,}
    {\it Banach algebra techniques in the theory of Toeplitz operators,}
    {Amer.\ Math.\ Soc., Providence, RI 1980.}
    
    
    \bibitem{HAKAN-KORENBLUM-ZHU}{H.~Hedenmalm, B.~Korenblum, K.~Zhu,}
    {\it Theory of Bergman spaces,}
    {Springer-Verlag, New York, 2000.}
   
   
    
    \bibitem{KORENBLUM-ZHU-AN APPLICATION OF TAUBERIAN}{B.~Korenblum, K.~Zhu,}
    {\it An application of Tauberian theorems to Toeplitz operators,}
    {J. Oper.\ Theory $\bm{33}$ (1995) 353--361.}
    
    \bibitem{LUECKING-INEQUALITIES ON BERGMAN SPACES}{D.~Luecking,}
    {\it Inequalities on Bergman spaces,}
    {Illinois J. Math.\ 25, $\bm{1}$ (1981), 1--11.}
    
	
	\bibitem{McDONALD-SUNDBERG-TOEPLITZ OPER. ON DISC}{G.~McDonald, C.~Sundberg,}
	{\it Toeplitz operators on the disc,}
	{Indiana Univ.\ Math.\ J. 28, $\bm{4}$ (1979), 595--611.}
	
	\bibitem{NIKOLSKI-TREATISE ON SHIFT}{N. K. Nikolski,}
	{\it Treatise on the Shift Operator: Spectral Function Theory,}
	{Springer, New York 1986.}
		
	\bibitem{PR-rkHs-Book}
	{V. I. Paulsen and M. Raghupathi}, 
	{ \it An Introduction to the Theory of Reproducing Kernel Hilbert Spaces}, 
	{Cambridge Studies in Advanced Mathematics}, vol. 152, Cambridge University Press, 			Cambridge, 2016. 
    \bibitem{Tolokonnikov-ESTIMATES IN THE CARLESON CORONA THM.}{V. A. Tolokonnikov,}
    {\it Estimates in the Carleson corona theorem, ideals of the algebra $H^{\infty}$, a 		problem of S.-Nagy,}
    {Z. Nauch.\ Sem.\ LOMI $\bm{113}$ (1981), 178--198.}
    
    
    \bibitem{WOLFF-COUNTEREXAMPLES}{T. H. Wolff,}
    {\it Counterexamples of two variants of Helson-Szeg\"{o} theorem,}
    {J. Anal.\ Math.\ $\bm{88}$ (2002), 41--62.}
    

    \bibitem{Yoneda-ON INVERTIBILITY OF BERGMAN TOEP. OPER.}{R.~Yoneda,}
    {\it Invertibility of Toeplitz operators on the Bergman spaces with harmonic 			     symbols,}
    {J. Math.\ Anal.\ Appl.\ $\bm{516}$ (2022):126515, no.\ 1.}



    \bibitem{Zhao_Zheng}{X.~Zhao, D.~Zheng,}
    {\it Invertibility of Toeplitz operators via Berezin transforms,}
    {J. Oper.\ Theory $\bm{75}$ (2016), no.\ 2, 475--495.}

	
	\bibitem{ZHU-OPER. THEORY IN FUNCTION SPACES}{K.~Zhu}
	{\it Operator theory in Function spaces,}
	{Amer.\ Math.\ Soc., Providence, 2007.}


\end{thebibliography}
\end{document}